\documentclass[11pt,reno]{amsart} 
\usepackage{amssymb}
\usepackage{amsmath}
\usepackage{enumitem}
\usepackage{mathrsfs}
\usepackage[english]{babel}
\usepackage[all]{xy}
\usepackage{hyperref}
\usepackage{marginnote}
\usepackage{comment}

\usepackage{stmaryrd}

\usepackage{fullpage}
\usepackage[margin=1.0in]{geometry}

\RequirePackage{mathrsfs} 

\newtheorem{theorem}{Theorem}[section]
\newtheorem{lemma}[theorem]{Lemma}
\newtheorem{proposition}[theorem]{Proposition}
\newtheorem{corollary}[theorem]{Corollary}
\newtheorem{conjecture}[theorem]{Conjecture}

\theoremstyle{definition}
\newtheorem*{ack}{Acknowledgments}
\newtheorem*{con}{Conventions}

\newtheorem{definition}[theorem]{Definition}

\numberwithin{equation}{section} \numberwithin{figure}{section}

 \DeclareMathOperator{\NS}{NS}

\DeclareMathOperator{\Sym}{Sym}
\DeclareMathOperator{\Spec}{Spec}

\DeclareMathOperator{\an}{an}

\DeclareMathOperator{\rank}{rank}
\DeclareMathOperator{\Hom}{Hom}

\DeclareMathOperator{\Jac}{Jac}

\newcommand{\Qbar}{\overline{\QQ}}

\newcommand{\Xbar}{\overline{X}}

\newcommand\ZZ{\mathbb{Z}}

\newcommand\QQ{\mathbb{Q}}

\newcommand\CC{\mathbb{C}}

\usepackage{color}

\definecolor{orange}{rgb}{1,0.5,0}

\title[Urata's theorem in the logarithmic case]{Urata's theorem in the logarithmic case and applications to integral points}

\author{Ariyan Javanpeykar}
\address{Ariyan Javanpeykar \\
Institut f\"{u}r Mathematik\\
Johannes Gutenberg-Universit\"{a}t Mainz\\
Staudingerweg 9, 55099 Mainz\\
Germany.}
\email{peykar@uni-mainz.de}

\author{Aaron Levin}
\address{Aaron Levin\\
  Department of Mathematics \\
Michigan State of University \\
 619 Red Cedar Road\\
 East Lansing, MI 48824\\
USA.}
\email{adlevin@math.msu.edu}

\begin{document}

 \begin{abstract}  
Urata showed that a pointed compact hyperbolic variety admits only finitely many maps from a pointed curve. We extend Urata's theorem to the setting of (not necessarily compact) hyperbolically embeddable  varieties. As an application, we show that a hyperbolically embeddable variety over a number field $K$ with only finitely many $\mathcal{O}_{L,T}$-points for any number field $L/K$ and any finite set of finite places $T$ of $L$ has, in fact, only finitely many points in any given $\mathbb{Z}$-finitely generated   integral domain of characteristic zero.   We use this latter result in combination with Green's criterion for hyperbolic embeddability to obtain novel finiteness results for integral points on symmetric self-products of smooth affine curves and on complements of large divisors in projective varieties. Finally, we use a partial converse to Green's criterion to further study hyperbolic embeddability (or its failure) in the case of symmetric self-products of curves.  As a by-product of our results, we obtain the first example of a smooth affine Brody-hyperbolic threefold over $\mathbb{C}$ which is not hyperbolically embeddable.
\end{abstract}

\maketitle

\thispagestyle{empty}

  \section{Introduction}

A point of view emphasized by Lang, beginning with at least his seminal work \cite{LangIHES}, posits that Diophantine statements involving rational points over number fields (or integral points over rings of integers) should continue to hold over arbitrary finitely generated fields over $\QQ$ (or their $\ZZ$-finitely generated subrings), and that this is the natural setting for such results. In this direction, after work of Siegel \cite{Siegel1}, Mahler \cite{Mahler1}, and Parry \cite{Parry} in the classical setting of rings of ($S$-)integers, Lang showed that the unit equation
\begin{align*}
u+v=1, \quad u,v\in A^*,
\end{align*}
has only finitely many solutions when $A$ is any $\ZZ$-finitely generated integral domain of characteristic zero. More generally, Lang \cite{LangIHES} proved Siegel's theorem \cite{Siegel2} on integral points on affine curves in this general setting (the unit equation corresponding to integral points on the affine curve $\mathbb{A}^1\setminus \{0,1\}$).  Lang first proved these results for rings of $S$-integers in number fields, and then used a specialization argument to deduce the general case.

Our main theorem provides a systematic approach to such results for integral points on affine varieties, frequently allowing one to extend results known over number fields (where there are many arithmetic tools available) to the aforementioned general setting.  More precisely, our method works for affine varieties that are {\it hyperbolically embeddable} (a purely complex-analytic notion due to Kobayashi).  Thus, our work provides another manifestation of the deep conjectural connections between arithmetic geometry and complex-analytic notions of hyperbolicity (see Conjecture \ref{LangConj}).  We note that it had already been suggested by Lang that hyperbolic embeddability may be the correct notion of hyperbolicity to use in relation to integral points on affine varieties (Conjecture \ref{LangConj2}).  Furthermore, a criterion of Green (Theorem \ref{thm:green}) makes it practical to check hyperbolic embeddability in many cases.

Before stating our precise results, we introduce some terminology.  Let $k$ be an algebraically closed field of characteristic zero.   By a \emph{variety} over $k$ we mean a finite type separated scheme over $k$.
A  variety 
  $X$  over $k$ is said to be \emph{arithmetically hyperbolic over $k$} 
if  there is a  $\ZZ$-finitely generated subring $A\subset k$ and a finite type separated $A$-scheme $\mathcal{X}$ with $\mathcal{X}_k \cong X$ over $k$ such that, for all $\ZZ$-finitely generated subrings $ A'\subset k$ containing $A$,  the set $\mathcal{X}(A')$  of $A'$-points  on $\mathcal{X}$ is finite.  If $X_L$ is arithmetically hyperbolic over $L$ for all algebraically closed field extensions $L\supset k$, then we say that $X$ is \emph{absolutely arithmetically hyperbolic}.  This extends to quasi-projective varieties Lang's notions \cite{Lang1} of the Mordell/Siegel property for projective/affine varieties; see also      \cite{Autissier1,  UllmoShimura, VojtaLangExc} and \cite{vBJK, JL, JLFano, JBook, JAut, JLitt, JXie}.
We refer to     \cite{Autissier1,   CLZ, Faltings2, FaltingsComplements, FaltingsLang, Levin, Moriwaki,  Vojta2,   VojtaSub} for   examples of arithmetically hyperbolic varieties.
 
Motivated by the previously mentioned principle, we are interested in the \emph{persistence} of the arithmetic hyperbolicity of a variety along field extensions. The formal statement we are interested in is the ``Persistence Conjecture'' for quasi-projective varieties (see also \cite[Conjecture~1.20]{vBJK}): 

\begin{conjecture}[Persistence Conjecture]
Let $k\subset L$ be an extension of algebraically closed fields of characteristic zero. 
If $X$ is arithmetically hyperbolic over $k$, then $X_L$ is arithmetically hyperbolic over $L$.
\end{conjecture}

If $X$ is a projective  variety over $k$, then this conjecture is a consequence of  a strong version of Lang-Vojta's conjectures as formulated in \cite[Section~12]{JBook}. Indeed, if $X$ is a projective arithmetically hyperbolic variety over $k$, then every subvariety of $X$ is of general type by this  conjecture. Then, it follows that every subvariety of $X_L$ is of general type (see for instance \cite{JV} for a proof of this well-known fact), so that (again by Lang-Vojta's conjecture), the variety $X_L$ is arithmetically hyperbolic over $L$.  

The Persistence Conjecture was first shown to hold for   algebraically hyperbolic projective varieties in \cite[Section~4]{JAut}. It was subsequently shown to hold for varieties which admit a quasi-finite morphism to a semi-abelian variety \cite[Theorem~7.4]{vBJK}, projective surfaces with non-zero irregularity \cite[Corollary~8.5]{vBJK}, varieties which admit a quasi-finite period map \cite{JLitt}, and certain moduli spaces of polarized varieties \cite{JSZ}.  

 Our main result (Theorem \ref{thm:hyp_emb1}) says that the Persistence Conjecture holds for hyperbolically embeddable smooth affine varieties. The proof relies on an extension to the affine case of a classical finiteness theorem of Urata for proper Kobayashi hyperbolic varieties. To explain this, we recall that, for proper varieties,  Kobayashi's  notion of   hyperbolicity is the ``right'' notion to consider. In this case, Lang \cite{Lang2} conjectured (see also \cite[Conj.~1.1]{JAut}) that Kobayashi hyperbolicity and arithmetic hyperbolicity coincide:
\begin{conjecture}
\label{LangConj}
Let $X$ be a projective variety over $k$.  Then $X$ is absolutely arithmetically hyperbolic if and only if for every subfield $k_0\subset \CC$, every embedding $k_0\to k$, and every variety $X_0$ over $k_0$ with $X\cong X_0\otimes_{k_0} k$, we  have that $X_{0,\CC}$ is Kobayashi hyperbolic.
\end{conjecture}

Lang has extended this conjecture to affine varieties, but notes that in this case there are \emph{a priori} three natural properties which may define ``hyperbolicity":
\begin{align}
\label{hyps}
\text{Brody hyperbolic }\Leftarrow \text{ Kobayashi hyperbolic }\Leftarrow \text{ hyperbolically embeddable}.
\end{align}

Choosing the strongest property, Lang conjectures (\cite[p.~86]{LangHyperbolic} and \cite[p.~225, Conj.~5.1]{LangSurvey}):

\begin{conjecture}
\label{LangConj2}
Let $X$ be an affine variety over $k$.   If there exists a subfield $k_0\subset \CC$, an embedding $k_0\to k$, and a variety $X_0$ over $k_0$ with $X\cong X_0\otimes_{k_0} k$ such that $X_{0,\CC}$ is hyperbolically embeddable, then $X$ is absolutely arithmetically hyperbolic.
\end{conjecture}

We stress that, given a Kobayashi hyperbolic  (resp. hyperbolically embeddable) variety over $\mathbb{C}$, it is not known whether the conjugates of $X$ are Kobayashi hyperbolic (resp. hyperbolically embeddable) in general.  

As a consequence of our main result, if $X$ is a smooth affine variety  over $\Qbar$, then to prove Conjecture \ref{LangConj2} it suffices to prove that $X$ is arithmetically hyperbolic over $\Qbar$ (i.e., it suffices to prove finiteness of integral points in the classical case, over rings of $S$-integers of number fields).

It is known that, in general, for quasi-projective varieties the three properties of \eqref{hyps} are inequivalent.  However, Lang has raised the question of which properties are equivalent in the affine case (e.g., \cite[p.~80]{LangHyperbolic} and \cite[pp.~225--226]{LangSurvey}). We give a partial answer to Lang's question, showing that there are smooth affine varieties which are Brody hyperbolic but not hyperbolically embeddable; see Theorem \ref{thm:finale2} below.

We follow \cite{KobayashiBook} and say that a reduced complex analytic space $\mathfrak{X}$ is Kobayashi hyperbolic if the Kobayashi pseudo-metric $d_{\mathfrak{X}}$ is a metric on $\mathfrak{X}$. A   variety $X$ over $\CC$  is Kobayashi hyperbolic if the analytification $X_{red}^{\an}$  of the reduced scheme $X_{red}$ is Kobayashi hyperbolic.

 Let $\overline{S}$ be a smooth projective variety over $\CC$ and let $S\subset \overline{S}$ be a dense open subscheme. We say that $S$ is \emph{hyperbolically embedded in $\overline{S}$} if the inclusion $S^{\an}\subset \overline{S}^{\an}$ of the complex analytic subspace $S^{\an}$ in $\overline{S}^{\an}$ is a hyperbolic embedding \cite[Chapter~3.3]{KobayashiBook}; note that this implies that $S$ is Kobayashi hyperbolic.
   
  A smooth affine variety $S$ over $\CC$ is \emph{hyperbolically embeddable} if there is a  projective  variety $\overline{S}$ and an open immersion $S\to \overline{S}$   such that $S$ is hyperbolically embedded in $\overline{S}$.    Note that, if $S$ is a hyperbolically embeddable variety,  then $S$ is Kobayashi (hence Brody) hyperbolic. A hyperbolically embeddable variety shares many common features with proper Kobayashi hyperbolic varieties, as is shown for instance in \cite{JKuch, KobayashiBook, Kwack}.

  We extend the above definitions to varieties over $k$ (an arbitrary algebraically closed field of characteristic zero). If $\overline{X}$ is a   proper variety over $k$ and $D$ is a closed subset of $\overline{X}$, then we will say that $X:=\overline{X}\setminus D$  is \emph{(weakly-)hyperbolically embedded in $\overline{X}$} if there is a subfield $k_0\subset k$, an embedding $k_0\to \mathbb{C}$,   a model $\overline{\mathcal{X}}$ for $\overline{X}$ over $k_0$, and a model $\mathcal{D}\subset \overline{\mathcal{X}}$ for $D\subset \overline{X}$ over $k_0$ such that   $ \overline{\mathcal{X}}_{\CC}\setminus \mathcal{D}_{\CC}$ is hyperbolically embedded in $\overline{\mathcal{X}}_{\CC}$.  This means, in particular, that $\mathcal{X}_{\CC}$ is Kobayashi hyperbolic (with respect to the fixed embedding $k_0\to \mathbb{C}$). We stress that this notion depends \emph{a priori} on the chosen embedding. As we are interested in algebraic properties of (weakly-)hyperbolically embedded varieties, this is harmless for our purposes.

\begin{theorem}[Main Result, I]\label{thm:hyp_emb1} Let $\overline{X}$ be a smooth projective      variety  over $k$,   and let $D\subset \overline{X}$ be a  divisor such that $X:=\overline{X}\setminus D$ is hyperbolically embedded in $\overline{X}$. 
 If $X$ is arithmetically hyperbolic over $k$, then   $X$ is \textbf{absolutely} arithmetically hyperbolic.
\end{theorem}

  To prove Theorem \ref{thm:hyp_emb1} we work with the notion of geometric hyperbolicity as defined in \cite{vBJK,  JBook, JLitt, JXie}. 
 A variety $X$  over $k$ is \emph{geometrically hyperbolic over $k$} if, for every smooth integral curve $C$ over $k$, every $c$ in $C(k)$, and every $x$ in $X(k)$, the set $\Hom_k((C,c),(X,x))$ of morphisms  $f:C\to X$ with $f(c)=x$ is finite.  As is explained in \cite[Remark~11.3]{JBook}, one may consider this property as a ``function field'' analogue of arithmetic hyperbolicity. In this paper we will exploit the properties of a variety which is simultaneously arithmetically hyperbolic and geometrically hyperbolic.

 To explain our proof of Theorem \ref{thm:hyp_emb1}, we note that 
 the following   result, obtained in \cite{JAut, JLitt},   implies that the Persistence Conjecture  holds for all varieties $X$ over $\Qbar$ such that $X_{\CC}$ is geometrically hyperbolic over $\CC$. This result has, for example, already been successfully applied to proving the Persistence Conjecture for varieties which admit a quasi-finite period map in \cite{JLitt}.

 \begin{theorem}\label{thm:geometricity_arhyp}
 Let $k$ be an algebraically closed subfield of $\CC$, and let $X$ be an arithmetically hyperbolic variety over $k$ such that  $X\otimes_k \CC$ is geometrically hyperbolic over $\CC$. Then,  the variety $X $ is \textbf{absolutely} arithmetically hyperbolic.
 \end{theorem}

  Theorem \ref{thm:geometricity_arhyp} implies that, to prove Theorem \ref{thm:hyp_emb1},    it suffices to show the geometric hyperbolicity of a hyperbolically embeddable variety. Towards this end, we first consider a classical finiteness theorem of Urata for \emph{proper} varieties; see      \cite[Theorem~5.3.10]{KobayashiBook} (or the original \cite{Urata}). 
  
 \begin{theorem}[Urata]\label{thm:urata}
 If $X$ is a  proper Brody hyperbolic variety over $\CC$, $Y$ is an integral variety over $\CC$, $y\in Y(\CC)$, and $x\in X(\CC)$, then    the set of morphisms $f:Y\to X$ with $f(y) = x$ is finite.    In particular, the variety $X$ is geometrically hyperbolic over $\CC$.
 \end{theorem}

Urata's proof uses the \emph{properness} of $X$ in a crucial way, and it is currently not known whether a Kobayashi hyperbolic quasi-projective variety over $\CC$ is geometrically hyperbolic over $\CC$. Indeed,  when transporting Urata's arguments to the non-proper setting one is confronted with several technical difficulties. Our next result shows that one can prove  the expected finiteness property for hyperbolically embedded varieties, under suitable assumptions.

\begin{theorem}[Main Result, II]\label{thm:urata_affine}
Let $\overline{X}$ be a smooth projective      variety  over $k$ and let $D\subset \overline{X}$ be a   divisor such that $X:=\overline{X}\setminus D$ is hyperbolically embedded in $\overline{X}$. 
If $Y$ is an integral variety over $k$, $y\in Y(k)$, and $x\in X(k)$, then the set of morphisms $f:Y\to X$ with $f(y) = x$ is finite.    In particular, the variety $X$ is geometrically hyperbolic over $k$.
\end{theorem}
  
   The proof of Theorem \ref{thm:hyp_emb1} is  thus  achieved by combining our    extension of Urata's theorem (Theorem \ref{thm:urata_affine}) with Theorem \ref{thm:geometricity_arhyp}. To prove Theorem \ref{thm:urata_affine} we use Kobayashi's result on the compactness of the moduli space of maps from a curve  to $X$ and a result of Pacienza-Rousseau on the logarithmic-algebraic hyperbolicity of $X$; see Section \ref{section:2}.
  
  As a first application of Theorem \ref{thm:hyp_emb1} we consider a finiteness result of the second-named author for varieties over number fields \cite[Th.~6.1A(b), Th.~6.2A(d)]{LevinAnnals}, \cite[Th.~1.4]{HL} (after work of Corvaja-Zannier \cite{CZ2,CZ3,CZ4}; see also work of Autissier \cite{Autissier1, Autissier2}), and extend the result to finitely generated fields.  
  
  \begin{theorem}[Main Result, III]\label{thm:III}
Let $m$ be a positive integer, let $X$ be a smooth projective connected variety over $\Qbar$ with $\dim X>1$, and let $D=\sum_{i=1}^rD_i$ be a sum of $r$ ample effective divisors on $X$ such that at most $m$ of the divisors $D_i$ meet in a point.  Suppose that $r\geq 2m \dim(X)$ (or $r\geq 5$ if $m=\dim X=2$).  Then the affine variety  $X\setminus D$ is absolutely arithmetically hyperbolic.
  \end{theorem}

More concretely, Theorem \ref{thm:III} says that, for every $\ZZ$-finitely generated integral domain $A$ of characteristic zero and any model $\mathcal{X}\setminus \mathcal{D}$ for $X\setminus D$ over $A$, the set  $(\mathcal{X}\setminus \mathcal{D})(A)$ of $A$-integral points on  this  model is finite. Assuming $\dim A=1$, this finiteness is proven in \cite{LevinAnnals} by the second-named author, building on seminal work of Corvaja-Zannier \cite{CZ2, CZ3, CZ4}. The proof in \cite{LevinAnnals} relies crucially on Schmidt's Subspace Theorem (as does Corvaja-Zannier's work). Our proof of Theorem \ref{thm:III}  combines the results of \cite{LevinAnnals} (over number fields) with the above ``complex-analytic'' results. It seems however interesting to  note that one could also instead appeal to the function field version of the Subspace Theorem \cite{AnWang, RuWang, Wang}, and ``re-do'' some of the arguments in \cite{LevinAnnals} to obtain Theorem \ref{thm:III}. This line of reasoning is pursued in \cite{Licht}.
 
Using results of Noguchi-Winkelmann \cite[Th.~7.3.4, Th.~9.7.6]{NWBook} (see also \cite{NW}), Theorem~\ref{thm:III} can be improved when the involved divisors generate a subgroup of small rank in the N\'eron-Severi group $\NS(X)$. 
 
\begin{theorem}[Main Result, IV]\label{thm:NW}
Let $X$ be a smooth projective connected variety over $\Qbar$, and let $D=\sum_{i=1}^rD_i$ be a sum of $r$ ample effective divisors on $X$ in general position.  Let $\rank \{D_i\}_{i=1}^r$ denote the (free) rank of the subgroup of $\NS(X)$ generated by the images of $D_1\ldots, D_r$. Suppose that $r\geq 2\dim X+\rank \{D_i\}_{i=1}^r$.  Then the affine variety  $X\setminus D$ is absolutely arithmetically hyperbolic.
\end{theorem} 
  
Lang conjectured that different notions of ``hyperbolicity'' for projective algebraic varieties should be equivalent. As discussed earlier, for affine varieties it is not so clear whether one should expect all notions of hyperbolicity to be equivalent. 
 In \cite{Levin} the second-named author studied the case of symmetric powers of smooth affine curves and obtained the following version of Lang-Vojta's conjecture; see \cite{AbrHar} for related results.

 \begin{theorem}[\cite{Levin}]\label{thm:LevinCompositio} Let $X$ be a smooth affine connected curve over $\Qbar$, and let $d\geq 1$ be an integer. Then  the following three statements are equivalent.
 \begin{enumerate}
 \item The smooth affine variety $\Sym^d_{X}$ is arithmetically  hyperbolic over $\Qbar$;
 \item  For every embedding $\Qbar \to \CC$, the complex-algebraic variety  $\Sym^d_{X_\CC}$ is Brody hyperbolic.
 \item For every dense open $U\subset X$,  every finite morphism  $U\to \mathbb{G}_{m,\Qbar}$  is of degree at least $d+1$.
 \end{enumerate}
 \end{theorem}

 We prove the following strengthening of Theorem \ref{thm:LevinCompositio}, under a suitable assumption on the boundary of $X$ in $\overline{X}$; see \cite{ShiffmanZaidenberg} for related results.

  \begin{theorem}[Main Result, V]  \label{thm:lv} Let $d\geq 1$ be an integer, and let $X$ be a smooth affine connected curve over $\Qbar$, with smooth projective model $\overline{X}$, such that $\# \overline{X}\setminus X \geq 2d$. Then the following statements are equivalent.
  \begin{enumerate}
  \item The  smooth   affine variety $\Sym^d_{X}$ is absolutely arithmetically hyperbolic.
 \item  For every embedding $\Qbar \to \CC$, the variety $\Sym^d_{X_{\CC}}$ is   hyperbolically embedded  in $\Sym^d_{\overline{X}_\CC}$.
 \item   For every dense open $U\subset X$, every finite morphism  $U\to \mathbb{G}_{m,\Qbar}$  is of degree at least $d+1$.
 \end{enumerate}
  \end{theorem}

If $\# \overline{X}\setminus X\geq 2d+1$, then it is easy to see that (3) in Theorem \ref{thm:lv} always holds.  In particular, we find (see Section \ref{sec:hypemb} for the validity over $\CC$):

\begin{corollary}\label{cor:symprods_are_hypemb}
Let $X$ be a smooth affine connected curve over $\mathbb{C}$ and let $\Xbar$ be its smooth projective compactification.  If $\#\Xbar\setminus X\geq 2d+1$ then $\Sym^d_X$ is hyperbolically embedded in $\Sym^d_{\Xbar}$.
\end{corollary}

The corollary may be viewed as a generalization of a well-known result, going back to Bloch, Cartan, and Dufresnoy (see also work of Green \cite{GreenHyp} and Fujimoto \cite{Fuji})  which asserts that the complement of $2n+1$ hyperplanes in general position in $\mathbb{P}^n$ is hyperbolically embedded in $\mathbb{P}^n$.  Indeed, if $X\subset \Xbar=\mathbb{P}^1$ and $r=\# \Xbar\setminus X$, then $\Sym^d_X$ is isomorphic to the complement of $r$ hyperplanes in general position in $\Sym^d_{\mathbb{P}^1} \cong \mathbb{P}^d$.  Thus, Corollary \ref{cor:symprods_are_hypemb} is a generalization of this result to symmetric powers of curves.

Quite interestingly,  the condition on the boundary in Theorem \ref{thm:lv} can be removed for the second symmetric self-product (see Corollary \ref{sym2}). Thus, we obtain the following strong version of Lang-Vojta's conjecture for the affine surface $\Sym^2_X$.

     \begin{theorem}[Main Result, VI]\label{thm4}  Let $ X$ be a smooth affine connected curve over $\Qbar$ with smooth projective model $\overline{X}$. Then the following statements are equivalent.
  \begin{enumerate}
  \item The  smooth   affine surface $\Sym^2_{X}$ is absolutely arithmetically hyperbolic.
 \item  For every embedding $\Qbar \to \CC$, the variety $\Sym^2_{X_{\CC}}$ is   hyperbolically embedded  in $\Sym^2_{\overline{X}_\CC}$.
 \item  For every dense open $U\subset X$, every finite morphism  $U\to \mathbb{G}_{m,\Qbar}$  is of degree at least $3$.
 \end{enumerate}
  \end{theorem}
 
The next result shows that the second and third statements of Theorem \ref{thm:lv} may be inequivalent when $d\geq 3$ and the numerical condition is not satisfied.

\begin{theorem}[Main Result, VII]\label{thm:finale}
For every $d\geq 3$, there exists a  smooth affine connected curve $X$ over $\Qbar$, with smooth projective model $\overline{X}$, such that $\#\Xbar\setminus X=2d-1$ and the following statements hold.
\begin{enumerate}
\item The smooth affine variety $\Sym^d_X$ is arithmetically hyperbolic over $\Qbar$.
\item For every embedding $\Qbar \to \CC$, the smooth affine variety $\Sym^d_{X_\CC}$ is Brody hyperbolic, but \textbf{not} hyperbolically embedded in  $\Sym^d_{\Xbar_{\CC}}$.
\item For every dense open $U\subset X$, every finite morphism  $U\to \mathbb{G}_{m,\Qbar}$  is of degree at least $d+1$.
\end{enumerate}
\end{theorem}

The fact that $\Sym^d_{X_\CC}$ is not hyperbolically embedded in $\Sym^d_{\Xbar_\CC}$ for $X$ as in Theorem \ref{thm:finale} still leaves open the possibility that the smooth affine variety $\Sym^d_{X_\CC}$ is hyperbolically embeddable.  However, by carefully choosing the curve $X$, we can show that $\Sym^3_{X_\CC}$ is not hyperbolically embeddable, thereby providing  the first examples of Brody hyperbolic smooth affine threefolds which are \textbf{not} hyperbolically embeddable.

\begin{theorem}[Main Result, VIII]\label{thm:finale2}
Let $\overline{X}$ be a smooth projective connected hyperelliptic curve of genus $g\geq 3$ over $\Qbar$, and let $\iota:\overline{X}\to \overline{X}$ be the hyperelliptic involution. Let $P_1, P_3, P_5$ be pairwise distinct non-Weierstrass points, let $P_2 :=\iota(P_1)$ and $P_4 := \iota(P_3)$, and write $X:= \overline{X}\setminus \{P_1,P_2, P_3,P_4,P_5\}$. Then the following statements hold.
\begin{enumerate}
\item The smooth affine threefold $\Sym^3_X$ is arithmetically hyperbolic over $\Qbar$.
\item For every embedding $\sigma:\Qbar \to \CC$, every projective  variety $Y$ over $\CC$ and every open immersion $X_{\sigma}\to Y$, the smooth affine threefold $\Sym^3_{X_\sigma}$ is Brody hyperbolic, but \textbf{not} hyperbolically embedded in  $Y$.
\end{enumerate}
\end{theorem}

\begin{con} Throughout this paper, we let $k$ be an algebraically closed field of characteristic zero. A variety over $k$ is a finite type separated scheme over $k$.  If $X$ is a variety over $k$ and   $A\subset k$ is a subring, then a \emph{model for $X$ over $A$} is a pair $(\mathcal{X},\phi)$ with $\mathcal{X}\to \Spec A$ a finite type separated scheme and $\phi:\mathcal{X}\otimes_A k \to X$ an isomorphism of schemes over $k$. We  omit $\phi$  from our notation.
If $X$ is a locally finite type scheme over $\CC$, we let $X^{\an}$ denote the associated complex analytic space. 
\end{con}

\begin{ack} We are grateful to Erwan Rousseau for helpful discussions.  The second-named author is supported in part by NSF grant DMS-1352407.
The first-named author gratefully acknowledges support from SFB/Transregio 45.
\end{ack}

  \section{Urata's theorem in the logarithmic case}\label{section:2}
  We follow  Kov\'acs-Lieblich's terminology  and  use the notion  of ``weak boundedness'' for quasi-projective schemes; see \cite{KovacsLieblich}.   
  
 \begin{definition}[Kov\'acs-Lieblich]\label{def:wb} Let $\overline{X}$ be a   projective scheme over $k$, let $\mathcal{L}$ be an ample line bundle on $\overline{X}$, and let $X\subset \overline{X}$ be a dense open subscheme. 
We say that $X$ is \emph{weakly bounded over $k$ in $\overline{X}$ with respect to $\mathcal{L}$} if, for every integer $g\geq 0$, and every $d\geq 0$, there  is a real number $\alpha(X,\overline{X}, \mathcal{L},g,d)$ such that, for every smooth projective connected curve $\overline{C}$ over $k$ of genus $g$, every dense open subscheme $C\subset \overline{C}$ with $\# (\overline{C}\setminus C) = d$, and   every morphism $f:C\to X$, the inequality
\[ 
\deg_{\overline{C}} \overline{f}^\ast \mathcal{L} \leq \alpha(X, \overline{X}, \mathcal{L}, g,d)
\] holds, where $\overline{f}:\overline{C}\to \overline{X}$ is the (unique) extension of $f:C\to X$.
 \end{definition}
 
 For $Y$ and $X$ projective varieties over $k$, we let $\underline{\Hom}_k(Y,X)$ be the locally finite type scheme parametrizing morphisms from $Y$ to $X$; see \cite{Nitsure}.  Recall that a projective scheme $X$ over $k$ is $1$-bounded over $k$ if, for every smooth projective connected curve $C$ over $k$, the scheme $\underline{\Hom}_k(C,X)$ is of finite type over $k$; see \cite[\S4]{JKa}.
 By \cite[Theorem~1.14]{JKa},  a \emph{projective} variety $X$  is weakly bounded (in itself) over $k$ if and only if it is $1$-bounded over $k$ (as defined in \cite[Definition~4.1]{JKa}).  
 Thus, the notion of weakly boundedness extends the notion of $1$-boundedness for projective varieties to quasi-projective varieties.

Our starting point is the following basic proposition. It says that the set of morphisms from a curve to a weakly bounded variety is (naturally) a quasi-compact (Zariski-)constructible subset of a certain Hom-scheme. More precisely, if $\overline{X}$ is a projective variety over $k$ and $\overline{C}$ is a smooth projective connected curve over $k$, then  we are interested in morphisms $C\to X$, where $C\subset \overline{C}$ is a dense open of $\overline{C}$ and $X\subset \overline{X}$ is a dense open of $\overline{X}$, respectively. The set $\Hom_k(C,X)$ of such morphisms is naturally a subset of the set  $\Hom_k(\overline{C},\overline{X})$ of morphisms $\overline{C}\to \overline{X}$ by the valuative criterion for properness. We   identify the latter set with the set of $k$-points of the scheme $\underline{\Hom}_k(\overline{C},\overline{X})$.

 \begin{proposition}\label{thm:kl}  \cite[Proposition~3.2]{JLitt}
 Let $\overline{X}$ be a projective variety over $k$, and  let $X\subset \overline{X}$ be a dense open subscheme. Let $\overline{C}$ be a smooth projective curve and let $C\subset\overline{C}$ be a dense open subscheme. If  there is an ample line bundle   $\mathcal{L}$  on $\overline{X}$ such that $X$ is weakly bounded over $k$ in $\overline{X}$ with respect to $\mathcal{L}$, then  
 $\Hom_k(C,X)$ is a   quasi-compact constructible subset of $\underline{\Hom}_k(\overline{C},\overline{X})(k)$.
 \end{proposition}
 
 Note that the property of being weakly bounded in $\overline{X}$ may  depend \emph{a priori} on the choice  of the ample line bundle $\mathcal{L}$ on $\overline{X}$ (and we ignore whether this is really the case). However, the conclusion of Proposition \ref{thm:kl} is \emph{independent} of $\mathcal{L}$. 

Our aim is to prove Urata's theorem in the logarithmic setting. Recall that Urata's theorem for a  proper Brody hyperbolic variety $X$ over $\mathbb{C}$ follows quite easily from the fact that the moduli space of maps $C\to X$ from any smooth proper curve $C$ is a proper scheme over $\CC$, and thus in particular of finite type.

In the logarithmic setting, when $X$ is hyperbolically embeddable, Kobayashi proves that the connected components of the space $\Hom_{\CC}((C,c),(X,x))$ of maps $f:C\to X$ with $f(c) =x$ are compact, but does not provide any information on the boundedness of this space. The additional ingredient we need in the ``logarithmic'' setting is the following theorem of Pacienza-Rousseau \cite{PacienzaRousseau}. This theorem actually follows  from their  extension of a theorem of Demailly \cite{Demailly} on the algebraic hyperbolicity of Brody hyperbolic varieties. Indeed, Demailly proved that, if $X$ is a Brody hyperbolic projective variety over $\CC$, then $X$ is algebraically hyperbolic over $\CC$ (in the sense of \cite{Demailly, JKa}). Pacienza-Rousseau's theorem is the natural extension of Demailly's theorem to the logarithmic setting.
 
 \begin{theorem}[Pacienza-Rousseau]\label{thm:pr} Let $\overline{X}$ be a  smooth projective variety over $\CC$, let $\mathcal{L}$ be an ample line bundle on $\overline{X}$, and let $D$ be a divisor on $\overline{X}$ with $X:=\overline{X}\setminus D$. If $X$ is hyperbolically embedded in $\overline{X}$, then $X$ is weakly bounded over $\CC$ in $\overline{X}$ with respect to $\mathcal{L}$.
 \end{theorem}
 \begin{proof} This follows from \cite[Theorem~5]{PacienzaRousseau} (although Pacienza-Rousseau require $D$ to have simple normal crossings, this is irrelevant to the proof). As our notation differs a bit from   \emph{loc. cit.}, we provide the details.
 
 Let $\overline{C}$ be a smooth projective connected curve of genus $g$ over $k$, let $C\subset \overline{C}$ be a dense open with $d:=\#(\overline{C}\setminus C)$, and let $f:C\to X$ be a morphism. Let $\overline{f}:\overline{C}\to \overline{X}$ be the unique extension of this morphism. Let $\overline{C}'$ be the normalization of the irreducible curve $\overline{f}(\overline{C})$, and note that $\overline{f}:\overline{C}\to \overline{X}$ factors over a morphism $\overline{f}':\overline{C}'\to \overline{X}$. Let $\nu:\overline{C}'\to \overline{f}(\overline{C})$ be the normalization map. We follow \cite{PacienzaRousseau} and let $i(\overline{f}(\overline{C}),D)$ be the number of elements in $\nu^{-1}(D)$. Note that, as $C$ lands in $X=\overline{X}\setminus D$, the inequality  $i(\overline{f}(\overline{C}),D) \leq d'$ holds, where $d' $ is the cardinality of the complement of the image of $C$ in $\overline{C}'$.  Therefore, 
as $X$ is hyperbolically embedded in the smooth projective variety $\overline{X}$ with $D:=\overline{X}\setminus X$ a divisor, it follows from  \cite[Theorem 5]{PacienzaRousseau} that there is an $\epsilon$ depending only on $\overline{X}$, $D$, and $\mathcal{L}$ such that
\[
\epsilon \deg_{\overline{C}'} \overline{f}'^{\ast} \mathcal{L} \leq 2g'-2 +i(\overline{f}(\overline{C}),D) \leq 2g'-2 +d',
\]  	 where $g'$ is the genus of $\overline{C}'$.  It follows from this inequality that 
\[ \deg_{\overline{C}}\overline{f}^\ast \mathcal{L} \leq  \epsilon^{-1} \deg(\overline{C}/\overline{C}')  (2g'-2+d') \leq \epsilon^{-1} (2g-2 +d).\] We define $\alpha := \epsilon^{-1}(2g-2+d)$ and see that, as required, the smooth affine variety $X$ is weakly bounded in $\overline{X}$ with respect to $\mathcal{L}$. 
 \end{proof}

 \begin{theorem} \label{thm:hypemb_is_geomhyp0}
Let $\overline{X}$ be a smooth projective      variety  over $\CC$  and let $D\subset \overline{X}$ be a  divisor such that $X:=\overline{X}\setminus D$ is hyperbolically embedded in $\overline{X}$.  
If $C$ is  a smooth quasi-projective connected curve over $\mathbb{C}$ with smooth projective model  $\overline{C}$, $ {c}\in \overline{C}(\CC)$, and $x$ in $\overline{X}(\CC)$, then the set of morphisms $\overline{f}:\overline{C}\to \overline{X}$ with $\overline{f}(C) \subset X$ and $\overline{f}( c)  = x$ is finite.
\end{theorem}
\begin{proof}  
If $D=\emptyset$, then $X =\overline{X}$ is a projective Kobayashi hyperbolic variety in which case the result follows from Urata's theorem \cite[Theorem~5.3.10]{KobayashiBook}. Thus,  we may and do assume $D\neq \emptyset$.  Let $\mathcal{L}$ be an ample line bundle on $\overline{X}$.
By Pacienza-Rousseau's theorem (Theorem \ref{thm:pr}), the variety $X$ is weakly bounded over $\CC$ in $\overline{X}$ with respect to $\mathcal{L}$. Therefore, by Proposition \ref{thm:kl}, the subset $\Hom_{\CC}(C,X)$ is a quasi-compact   constructible subset of  $\underline{\Hom}_{\CC}(\overline{C},\overline{X})(\CC)$. 

 Now, we consider the quasi-compact constructible subset $\mathcal{F}:=\Hom_{\CC}(C,X)$ as a subset of the complex-analytic space $\underline{\Hom}_{\CC}(\overline{C}^{\an}, \overline{X}^{\an}) = \underline{\Hom}_{\CC}(\overline{C},\overline{X})^{\an}$. Let $\overline{\mathcal{F}}$ be the analytic-closure of $\mathcal{F}$ in the analytification $\underline{\Hom}_{\CC}(\overline{C},\overline{X})^{\an}$ of the scheme $\underline{\Hom}_{\CC}(\overline{C},\overline{X})$, and note that $\overline{\mathcal{F}}$ has only finitely many connected components (as $\mathcal{F}$ is quasi-compact). By a theorem of Kobayashi \cite[Theorem 6.4.10.(4)]{KobayashiBook}, the universal evaluation map
 \[
 \overline{C}\times \overline{\mathcal{F}} \to \overline{C}\times  \overline{X}, \quad (c,f)\mapsto (c,f(c))
 \] has finite fibers. This implies  that $\overline{C}\times \mathcal{F} \to \overline{C}\times \overline{X}$ has finite fibers, thereby proving the required finiteness statement. 
\end{proof}

 We can now deduce the desired finiteness statement (i.e., Urata's finiteness theorem in the logarithmic case) over arbitrary algebraically closed fields of characteristic zero.
 
 \begin{theorem} \label{thm:hypemb_is_geomhyp}
Let $\overline{X}$ be a smooth projective      variety  over $k$,      and let $D\subset \overline{X}$ be a  divisor such that $X:=\overline{X}\setminus D$ is hyperbolically embedded in $\overline{X}$.  
Then,  $X$ is geometrically hyperbolic over $k$.  
\end{theorem}
\begin{proof} By \cite[Lemma~2.4]{JLitt}. we may and do assume that  $k=\CC$ .  Then, 
the statement  clearly follows from Theorem \ref{thm:hypemb_is_geomhyp0}.
\end{proof}

 \begin{lemma}\label{lem1} Assume that $k$ is uncountable.
 Let $X$ be a geometrically hyperbolic variety over $k$. Then, for every integral variety $Y$ over $k$, every $y$ in $Y(k)$, and every $x$ in $X(k)$, the set of morphisms $f:Y\to X$ with $f(y)=x$ is finite. 
 \end{lemma}
 \begin{proof} We argue by contradiction.    Let $Y$ be a $d$-dimensional integral variety over $k$, let $y\in Y(k)$, let $x\in X(k)$, and let 
 \[
 f_1, f_2, \ldots \in \Hom_{k}((Y,y),(X,x))
 \] be a sequence of pairwise distinct morphisms  from $Y$ to $X$ which send $y$ to $x$. Note that $d>1$ by our assumption that $X$ is geometrically hyperbolic over $k$. For $n,m\in \mathbb{Z}_{\geq 1}$, we define $Y^{n,m} = \{p \in Y(k) \ | \ f_n(p) = f_m(p)\}$. Note that, for every $n\neq m$, the subset $Y^{n,m}\subset Y(k)$ is a proper closed subset. In particular, by the uncountability of $k$, there is a point $P\in Y(k)$ in the complement of the union $\cup_{n\neq m} Y^{n,m}$. Let $C$ be a smooth integral curve over $k$ in $Y$ containing $y$ and $P$. Then, the morphisms $f_i|_C:C\to X$ are pairwise distinct morphisms sending $y$ to $x$. This contradicts the geometric hyperbolicity of $X$, and concludes the proof.
 \end{proof}
 
\begin{proof} [Proof of Theorem \ref{thm:urata_affine}] We may and do assume that $k$ is uncountable. Then, 
by  Theorem \ref{thm:hypemb_is_geomhyp}, the variety $X$ is geometrically hyperbolic over $k$, so that the result follows from Lemma \ref{lem1}.
\end{proof}

 \begin{proof}[Proof of Theorem \ref{thm:hyp_emb1}]
 Combine Theorem \ref{thm:geometricity_arhyp} and Theorem \ref{thm:urata_affine}.
 \end{proof}

  \section{Complements of large divisors and a result of Noguchi-Winkelmann} 
  We start with a     result  of the second-named author \cite[Section~6]{LevinAnnals}; see also    \cite{Autissier1, Autissier2, CLZ, CZ, HL}.
  We   follow the   setup  in \cite[Section~4]{LevinAnnals}, and  consider the following data fixed.
  \begin{itemize}
  \item A positive integer $m$ and a positive integer $r$;
  \item A smooth projective variety $X$ over $k$ with $\dim X>1$;
  \item An effective  divisor $D = \sum_{i=1}^r D_i$ on $X$ such that at most $m$ of the divisors $D_i$ meet in a point.
  \end{itemize}    
  
  In this situation,   the arithmetic hyperbolicity (when $k=\Qbar$) and the analytic hyperbolicity of the complement of $D$ in $X$ were verified in \cite{LevinAnnals}, under suitable ``positivity'' assumptions on $D$.
  
  \begin{theorem}[Levin]\label{thm:Levin_divs} Assume that $r\geq 2m \dim(X)$ (or $r\geq 5$ if $m=\dim X=2$) and that, for every $i$, the divisor $D_i$ is ample. Then 
the following statements hold. 
\begin{enumerate}
\item If $k=\Qbar$, then $X\setminus D$ is arithmetically hyperbolic over $\Qbar$.
\item If $k=\CC$, then the smooth affine  variety $X \setminus D $ is hyperbolically embedded in $X$.
\end{enumerate}
  \end{theorem}  
  \begin{proof}
  This is \cite[Theorem~6.1A, Theorem~6.2A, Theorem~6.1B and Theorem~6.2B]{LevinAnnals}, with a slight improvement coming from \cite{Autissier1} (see \cite[p.~2]{HL}).  
  \end{proof}
  
 We establish the following extensions of Theorem \ref{thm:Levin_divs} (with the notation as stated above Theorem \ref{thm:Levin_divs}).

 \begin{theorem}[Geometric hyperbolicity]\label{thm:geomhyp91} Assume that $r\geq 2m \dim(X)$ (or $r\geq 5$ if $m=\dim X=2$) and that, for every $i$, the divisor $D_i$ is ample. Then, for every integral variety $Y$ over $k$, every $y$ in $Y(k)$, and every $x$ in $(X\setminus D)(k)$, the set of morphisms $f:Y\to X\setminus D$ with $f(y)=x$ is finite.  In particular, the variety  $X\setminus D$ is   geometrically hyperbolic over $k$.   
 \end{theorem}
 \begin{proof} To prove the geometric hyperbolicity of $X\setminus D$, by \cite[Lemma~2.4]{JLitt}, we may and do assume that $k$ is the field of complex numbers.  Then, by the second part of Theorem \ref{thm:Levin_divs}, the (complex algebraic) variety $X\setminus D$ is hyperbolically embedded in $X$. Thus, the required finiteness statement follows from the logarithmic version of Urata's theorem (Theorem \ref{thm:urata_affine}).
 \end{proof}
 
Now using this result and the first part of Theorem \ref{thm:Levin_divs} combined with Theorem \ref{thm:geometricity_arhyp} (or alternatively,  simply applying Theorem \ref{thm:hyp_emb1} with Theorem \ref{thm:Levin_divs}) we find:

 \begin{theorem}[Arithmetic hyperbolicity] Assume that $r\geq 2m \dim(X)$ (or $r\geq 5$ if $m=\dim X=2$) and that, for every $i$, the divisor $D_i$ is ample. If $k=\Qbar$ and $k\subset L$ is an extension of algebraically closed fields, then  $X_L\setminus D_L$ is   arithmetically hyperbolic over $L$.
 \end{theorem}

The proof of the extension (Theorem \ref{thm:NW}) of Noguchi-Winkelmann's results follows in the same manner from:

\begin{theorem}[Noguchi-Winkelmann {\cite[Th.~7.3.4, Th.~9.7.6]{NWBook}} (see also \cite{NW})]  Let $X$ be a smooth projective connected variety over $k$, and let $D=\sum_{i=1}^rD_i$ be a sum of $r$ ample effective divisors on $X$ in general position.  Let $\rank \{D_i\}_{i=1}^r$ denote the (free) rank of the subgroup of $\NS(X)$ generated by the images of $D_1\ldots, D_r$. Suppose that $r\geq 2\dim X+\rank \{D_i\}_{i=1}^r$. Then 
the following statements hold. 
\begin{enumerate}
\item If $k=\Qbar$, then $X\setminus D$ is arithmetically hyperbolic over $\Qbar$.
\item If $k=\CC$, the smooth affine  variety $X \setminus D $ is hyperbolically embedded in $X$.
\end{enumerate}
  \end{theorem}

\section{Hyperbolic embeddings of symmetric products}\label{sec:hypemb}
    In this section we prove that symmetric products of smooth affine curves are hyperbolically embedded in the symmetric product of their smooth projective model, under suitable assumptions; see Corollaries \ref{cor:symprods_are_hypemb} and \ref{cor:hypemb}. 
      To prove these results, we will use the following theorem of Green.  
    \begin{theorem}[Green]\label{thm:green}
  Let $Z$ be a smooth projective variety and let $D$ be the union of Cartier divisors $D_1, \dots, D_m$. Then $Y = Z \setminus D$ is hyperbolically embedded in $Z$ if the following two conditions are satisfied:
    \begin{enumerate}
  \item $Y$ is Brody hyperbolic;
  \item for any partition of indices $I \cup J = \{ 1, \dots, m \}$, the variety $\cap_{i \in I} D_i \setminus \cup_{j \in J} D_j$ is Brody hyperbolic. 
  \end{enumerate}
    \end{theorem}
  
\begin{proof}
See \cite[Theorem 3.6.13]{KobayashiBook} (or the original \cite{GreenHyp}). 
\end{proof}

We will also be interested in   showing that certain symmetric products of curves are not hyperbolically embedded in their canonical model. To do so, we will use the following (partial) converse to Green's theorem due to Noguchi and Winkelmann \cite[Th.~7.2.13]{NWBook} (see also Zaidernberg's partial converse in \cite[Theorem 3.6.18]{KobayashiBook} and \cite{Zaidenberg1, Zaidenberg2}). (The reason we say ``partial converse'' is because of the additional general position assumption.)
 
  \begin{theorem}\label{thm:greenconverse}
  Let $Z$ be a smooth projective variety and let $D$ be the union of Cartier divisors $D_1, \dots, D_m$ in general position on $Z$. If $Y = Z \setminus D$ is hyperbolically embedded in $Z$, then
  \begin{enumerate}
  \item $Y$ is Brody hyperbolic;
  \item for any partition of indices $I \cup J = \{ 1, \dots, m \}$, the variety $\cap_{i \in I} D_i \setminus \cup_{j \in J} D_j$ is Brody hyperbolic. 
  \end{enumerate}
 \end{theorem}

We start with using Green's theorem and its partial converse to prove the following result. 

\begin{theorem}
\label{symhyp}
Let $X$ be a smooth affine connected curve over $\mathbb{C}$ and let $\Xbar$ be its smooth projective compactification.  Let $d$ be a positive integer.  Then $\Sym^d_X$ is hyperbolically embedded in $\Sym^d_{\Xbar}$ if and only if for every subset $T\subset \Xbar\setminus X$ with $0\leq |T|<d$, the variety $\Sym^{d-|T|}_{X\cup T}$ is Brody hyperbolic. 
\end{theorem}

\begin{proof}
Let $\Xbar\setminus X=\{P_1,\ldots, P_r\}$.  Let $\psi:\Xbar^d\to \Sym^d_{\Xbar}$ be the natural map and let $\pi:\Xbar^d\to \Xbar$ be one of the natural projections.  Let $D_j=\psi_*\pi^*P_j$, $j=1,\ldots, r$.  Then $D=\sum_{j=1}^rD_j$ is a normal crossings divisor on $\Sym^d_{\Xbar}$  and $\Sym^d_X=\Sym^d_{\Xbar}\setminus D$.  In particular, the divisors $D_1,\ldots, D_r$ are in general position.

 Let $\emptyset \subset I\subset\{1,\ldots, r\}$ and let $J=\{1,\ldots, r\}\setminus I$.  Then $\cap_{i\in I}D_i\cong \Sym^{d-|I|}_{\Xbar}$ if $|I|\leq d$ and $\cap_{i\in I}D_i=\emptyset$ if $|I|>d$.  In the first case, if $|I|\leq d$, then 

\begin{align*}
\bigcap_{i\in I}D_i\setminus \bigcup_{j\in J}D_j\cong \Sym^{d-|I|}_{\Xbar\setminus \cup_{j\in J}\{P_j\}}=\Sym^{d-|I|}_{\cup_{i\in I}\{P_i\}\cup X}.
\end{align*}

Now the result follows from Green's theorem (Theorem \ref{thm:green}) and its (partial) converse (Theorem \ref{thm:greenconverse}), noting that when $|I|=d$, $\bigcap_{i\in I}D_i$ is a point.
\end{proof}

As we will show later, Brody hyperbolicity and being hyperbolically embeddable are not equivalent notions in general. However, it seems reasonable to suspect that in the case of smooth affine surfaces these two notions do in fact coincide.  The following corollary is in accordance with this expectation.

\begin{corollary}
\label{sym2}
Let $X$ be a smooth affine connected curve over $\mathbb{C}$ and let $\Xbar$ be its smooth projective compactification.  Then $\Sym^2_X$ is hyperbolically embedded in $\Sym^2_{\Xbar}$ if and only if $\Sym^2_X$ is Brody hyperbolic.
\end{corollary}

\begin{proof}
By Theorem \ref{symhyp} with $d=2$, the corollary is equivalent to showing that if $\Sym^2_X$ is Brody hyperbolic then $X\cup \{P\}$ is Brody hyperbolic for any point $P\in \Xbar\setminus X$. To do so, let $g(\Xbar)$ denote the genus of $\Xbar$.  If $g(\Xbar)\geq 2$ then this is vacuous as $\Xbar$ is Brody hyperbolic.  If $g(\Xbar)=1$ then $X\cup \{P\}$ is Brody hyperbolic unless $X\cup \{P\}=\Xbar$ and $\Xbar\setminus X$ consists of a single point.  But in this case $\Sym^2_X$ is not Brody hyperbolic.  If $X$ is rational, then $\Sym^2_X$ is Brody hyperbolic if and only if $\#\Xbar\setminus X\geq 5$, which implies that $X\cup \{P\}$ is Brody hyperbolic for any point $P\in \Xbar\setminus X$.
\end{proof}

\begin{corollary}
\label{corsym}
Let $D=\Xbar\setminus X$ and let $d$ be a positive integer.  If $\deg D\geq d$ then $\Sym^d_X$ is hyperbolically embedded in $\Sym^d_{\Xbar}$ if and only if for every subset $T\subset D$, $t=|T|$, with $0\leq t<d$, there is no finite morphism $f:\overline{X}\to \mathbb{P}^1_{\CC}$ of degree at most $d-t$ such that $f(D\setminus T)\subset \{0,\infty\}$.
\end{corollary}
\begin{proof} 
Suppose that $\deg D\geq d$ and let $T\subset D$, $t=|T|$, with $0\leq t<d$.   Then, by the main (analytic) result of \cite{Levin},  $\Sym^{d-t}_{X\cup T}$ is Brody hyperbolic if and only if  there is no finite morphism $f:\overline{X}\to \mathbb{P}^1_{\CC}$ of degree at most $d-t$ such that $f(D\setminus T)\subset \{0,\infty\}$.  Thus, the result follows from Theorem \ref{symhyp}.
\end{proof}

As a consequence, we obtain Corollary \ref{cor:symprods_are_hypemb} from the introduction.

\begin{proof}[Proof of Corollary \ref{cor:symprods_are_hypemb}]
Since $(2d+1)-t> 2(d-t)$ for $t\geq 0$, this follows immediately from Corollary \ref{corsym} as for any morphism $\phi:\Xbar\to\mathbb{P}^1$, at most $2\deg \phi$ points of $\Xbar$ can map to $\{0,\infty\}$.
\end{proof}

Essentially the same proof gives:

\begin{corollary}\label{cor:hypemb}  
Let $X$ be a smooth affine connected curve over $\mathbb{C}$ and let $\Xbar$ be its smooth projective compactification.  If $\#\Xbar\setminus X\geq 2d$ and $\Sym^d_X$ is Brody hyperbolic, then $\Sym^d_X$ is hyperbolically embedded in $\Sym^d_{\Xbar}$.
\end{corollary}

  \section{Integral points on symmetric products}
  
    We follow \cite{JBook, JKa, JV} and say that a variety $V$ over $k$ is \emph{groupless over $k$} if, for every finite type connected group scheme $G$ over $k$, every morphism $G\to V$ is constant. Note that Lang refers to such varieties as being ``algebraically-hyperbolic''; see \cite{Lang2}. If $V$ is affine, then  $V$ is groupless if and only if every morphism $\mathbb{G}_{m,k}\to V$ is constant  \cite[Lemma~2.5]{JKa}.  Moreover, by \cite[Lemma~2.3]{JKa}, if $L/k$ is an extension of algebraically closed fields, then $V$ is groupless over $k$ if and only if $V_L$ is groupless over $L$. 
    
    We are concerned with symmetric products of smooth affine connected curves.  
Let $X$ be a smooth affine connected curve over $k$,  and let $d\geq 1$. An obvious obstruction to the arithmetic hyperbolicity of $\Sym^d_X$ is the existence of a non-constant morphism $\mathbb{G}_{m,k}\to \Sym^d_X$, where $\mathbb{G}_m $ denotes the multiplicative group scheme over $\mathbb{Z}$. That is, if $\Sym^d_X$ is arithmetically hyperbolic over $k$, then $\Sym^d_X$ is groupless over $k$ (see also \cite[Proposition~3.9]{JAut}).   Lang-Vojta's conjectures predict   that  the affine variety $\Sym^d_X$ is arithmetically hyperbolic over $k$ if (and only if) it is groupless over $k$. To make this more concrete,  we now give an explicit criterion for the grouplessness of $\Sym^d_X$. 

\begin{proposition}\label{prop:grouplessness}
There is a non-constant morphism $\mathbb{G}_{m,k}\to \Sym^d_X$  if and only if there is a dense open $U\subset X$ and a finite morphism $U\to \mathbb{G}_{m,k}$ of degree at most $d$. 
\end{proposition}
\begin{proof}
Let $\overline{X}$ be the smooth projective model for $X$ over $k$, and let $\{P_1,\ldots, P_r\} :=\overline{X}\setminus X$.  Since abelian varieties do not contain rational curves, if there is a non-constant morphism $\varphi:\mathbb{G}_{m,k}\to \Sym^d_X$, then the image of $\varphi$ must lie in a fiber of the Abel-Jacobi map on $\Sym^d_{\overline{X}}$.  This fiber is isomorphic to some projective space $\mathbb{P}^n_k$ and corresponds to some complete linear system on the curve $\overline{X}$ (interpreting points of $\Sym^d_{\overline{X}}$ as degree $d$ effective divisors).  The intersection of this fiber with $\Sym^d_X$ is a complement of hyperplanes $H_1,\ldots, H_r$ in $\mathbb{P}^n_k$ (one hyperplane $H_i$ for each point $P_i$).  If there exists a line $L$ intersecting $\cup H_i$ in two or fewer points, then the line $L$ corresponds to a linear system giving a finite morphism $f:\overline{X}\to \mathbb{P}^1_k$ of degree at most $d$ which maps $\{P_1,\ldots, P_r\}$ to a set of (at most) two points (which we may take to be $\{0,\infty\}$). It follows from \cite[Theorem~4.1]{Levin} that there is a dense open $U\subset X$ and a finite morphism $U\to \mathbb{G}_{m,k}$ of degree at most $d$.  Thus, it suffices to show that if there exists a non-constant morphism $\mathbb{G}_{m,k}\to \mathbb{P}^n_k\setminus \cup_{i=1}^rH_i$ (i.e., the complement of $H_1,\ldots, H_r$ is non-groupless), then there exists a linear such $\mathbb{G}_{m,k}$ (equivalently, a line in $\mathbb{P}^n_k$ intersecting $\cup H_i$ in two or fewer points).  

Suppose now that there exists a non-constant morphism $\varphi:\mathbb{G}_{m,k}\to \mathbb{P}^n_k\setminus \cup_{i=1}^rH_i$.  By replacing $\mathbb{P}^n_k$ by an appropriate linear space, we may assume that the image of $\mathbb{G}_{m,k}$ is not contained in any hyperplane.  Let $L_1,\ldots, L_r$ be linear forms over $k$ defining $H_1,\ldots, H_r$.  If $L_1,\ldots, L_r$ are linearly independent then $r\leq n+1$ and it's easy to see that the desired line exists (take any appropriate line containing a point in $\cap_{i=1}^{r-1}H_i$). Otherwise, let $L_{i_1},\ldots, L_{i_m}$ be linearly dependent over $k$ with $m$ minimal.  Let $\phi_j=\frac{L_{i_j}}{L_{i_m}}\circ \varphi$, $j=1,\ldots, m$.  Then $\phi_j$ and $1/\phi_j$ are regular functions on $\mathbb{G}_{m,k}$, and $\phi_j$ may be identified with $c_jt^{n_j}\in k(t)$ for some constant $c_j\in k^*$ and $n_j\in\mathbb{Z}$.  Since the image of $\varphi$ is not contained in any hyperplane, from the minimality of $m$, all of the powers $n_j$ are distinct. However, this clearly contradicts that $\phi_1,\ldots, \phi_m$ are linearly dependent over $k$ (this last argument is an elementary case ($g=0$, $|S|=2$) of the function field $S$-unit equation height inequality of Mason \cite{Mason}, Brownawell-Masser \cite{BM}, and Voloch \cite{Voloch}).

Conversely, the existence of an open $U\subset X$ and finite morphism $U\to \mathbb{G}_{m,k}$ of degree at most $d$ obviously imply the existence of  a non-constant morphism $\mathbb{G}_{m,k}\to \Sym^d_X$. 
\end{proof}

Proposition \ref{prop:grouplessness} says that $\Sym^d_X$ is not groupless over $k$ if and only if there exists a dense open $U\subset X$ and a finite morphism $U\to \mathbb{G}_{m,k}$ of degree at most $d$. 
In \cite{Levin} the second-named author showed that, if $k=\Qbar$, then the obvious obstruction to the arithmetic hyperbolicity of $\Sym^d_X$ (i.e., its non-grouplessness) is in fact the only one. That is, the variety $\Sym^d_X$ is arithmetically hyperbolic over $\Qbar$ if and only if $\Sym^d_X$ is groupless over $\Qbar$. We  now prove the following extension of Theorem \ref{thm:LevinCompositio}, as stated in the introduction.

  \begin{proof}[Proof of Theorem \ref{thm:lv}] Let $X$ be a smooth affine connected curve over $\Qbar$ with smooth projective model $\overline{X}$. Let $d\geq 1$ and assume that $2d \leq \#(\overline{X}\setminus X)$.  Note that $(3)$ holds if and only if  $\Sym^d_X$ is groupless by Proposition \ref{prop:grouplessness}. From this it is clear that $(1)\implies (3)$, as arithmetically hyperbolic varieties are groupless \cite[Proposition~3.9]{JAut}. Similarly, as Brody hyperbolic varieties are (obviously) groupless, we see that $(2)\implies (3)$.

Assume $(3)$ holds, i.e., the affine variety $\Sym^d_X$ is groupless over $\Qbar$. Then, by \cite{Levin}, $\Sym^d_X$ is arithmetically hyperbolic over $\Qbar$ and the variety $\Sym^d_{X_\CC}$ is Brody hyperbolic. Therefore,  by  Corollary \ref{cor:hypemb}, the variety $\Sym^d_{X_\CC}$ is hyperbolically embedded in $\Sym^d_{\overline{X}_\CC}$.  This shows that $(2)$ holds.  In particular, by our main result (Theorem \ref{thm:hyp_emb1}), the Persistence Conjecture holds for $\Sym^d_X$. Thus, as $\Sym^d_X$ is arithmetically hyperbolic over $\Qbar$, it follows that for every algebraically closed field $k$ of characteristic zero, the variety $\Sym^d_{X_k}$ is arithmetically hyperbolic over $k$. This  shows that $(1)$ also holds, and concludes the proof.
  \end{proof}
  
  \begin{proof}[Proof of Theorem \ref{thm4}] Let $X$ be a smooth affine connected curve over $\Qbar$. To prove the theorem, 
we argue as in the proof of Theorem \ref{thm:lv}. Indeed,   since $(1)\implies (3)$ and $(2)\implies (3)$, we may and do assume that $(3)$ holds. Then, by   \cite{Levin}, the variety $\Sym^2_{X_\CC}$ is Brody hyperbolic  and $\Sym^2_{X}$ is arithmetically hyperbolic over $\Qbar$. As $\Sym^2_{X_\CC}$ is Brody hyperbolic, it follows from  Corollary  \ref{sym2} that  the surface $\Sym^2_{X_\CC}$ is hyperbolically embedded in $\Sym^2_{\overline{X}_\CC}$. This shows that $(2)$ holds.  Next, by Theorem \ref{thm:hyp_emb1}, the Persistence Conjecture holds for  the surface $\Sym^2_{X}$, so that $\Sym^2_X $ is absolutely arithmetically hyperbolic. This shows that $(1)$ holds, as required.
  \end{proof}

  \section{Non-hyperbolic embeddings of Brody hyperbolic symmetric products} \label{section:final}

In this section we prove Theorems \ref{thm:finale} and \ref{thm:finale2}. First, 
we show that for $d\geq 3$ the quantity $2d$ in   Corollary \ref{cor:hypemb} is sharp (see Corollary~\ref{sym2} for $d=2$).

\begin{proof}[Proof of Theorem \ref{thm:finale}]
Let $d\geq 3$.  Let $\Xbar$ be a smooth projective curve over $\Qbar$ of genus $g(\Xbar)>(d-2)(d-1)$ and gonality $d-1$ (such curves are easily constructed).  Let $\phi:\Xbar\to\mathbb{P}^1$ be a morphism of degree $d-1$, which after an automorphism of $\mathbb{P}^1$ we can assume is unramified above $0$ and $\infty$.  Let $P\in \Xbar\setminus \phi^{-1}(\{0,\infty\})$ and let $D=\phi^{-1}(\{0,\infty\})\cup \{P\}$. Let $X=\Xbar\setminus D$ and note that $\#\Xbar\setminus X=2d-1$.  By Corollary \ref{corsym} with $T=\{P\}$, for every embedding $\Qbar \to \CC$, we see that $\Sym^d_{X_\CC}$ is not hyperbolically embedded in $\Sym^d_{\Xbar_\CC}$.  On the other hand, it follows from Castelnuovo's inequality \cite[p.~366]{ACGH} and our assumptions on the genus and gonality of $\Xbar$ that there does not exist a morphism $\Xbar\to\mathbb{P}^1$ of degree $d$.  Since $\#\Xbar\setminus X=2d-1>2(d-1)$, this immediately implies that $(3)$ holds. Therefore, by Theorem \ref{thm:LevinCompositio}, $\Sym^d_X$ is arithmetically hyperbolic over $\Qbar$ (so that $(1)$ holds) and, for every embedding $\Qbar \to \CC$, the variety $\Sym^d_{X_\CC}$ is Brody hyperbolic (so that $(2)$ holds). This concludes the proof.
\end{proof}

\begin{theorem}\label{thm:final_hypemb}
Let $\overline{X}$ be a smooth projective connected hyperelliptic curve of genus $g\geq 3$ over $\CC$, and let $\iota:\overline{X}\to \overline{X}$ be the hyperelliptic involution. Let $P_1, P_3, P_5$ be pairwise distinct non-Weierstrass points, let $P_2 :=\iota(P_1)$ and $P_4 := \iota(P_3)$, and write $X:= \overline{X}\setminus \{P_1,P_2, P_3,P_4,P_5\}$.   Then $\Sym^3_X$ is not hyperbolically embeddable.
\end{theorem}
\begin{proof}
Let $Y$  be a projective variety over $\mathbb{C}$ and let $\Sym^3_{X}\subset Y$ be a hyperbolic embedding (so that, in particular, the threefold $\Sym^3_{X}$ is Kobayashi hyperbolic). Consider the natural embedding $\Sym^3_X\subset \Sym^3_{\Xbar}$, and note that its complement is a normal crossings divisor. Therefore, as $\Sym^3_X$ is hyperbolically embedded in $Y$, the extension theorem of Kiernan-Kobayashi-Kwack \cite[Theorem 6.3.9]{KobayashiBook} implies that the identity map 
$\Sym^3_X\to \Sym^3_X$ extends to a morphism $\Sym^3_{\Xbar} \to Y$. 

As in the proof of Theorem \ref{symhyp}, if $D_1,\ldots, D_5$ are the effective divisors on $\Sym^3_{\Xbar}$ naturally corresponding to the points $P_1,\ldots, P_5$, then
\begin{align*}
D_5\setminus \cup_{j=1}^4D_j\cong \Sym^2_{\overline{X}\setminus \{P_1,P_2, P_3,P_4\}}.
\end{align*}
The unique $g^1_2$ on $\overline{X}$ corresponds to a $\mathbb{P}^1$ in $\Sym^2_{\overline{X}}$, and the intersection with $\Sym^2_{\overline{X}\setminus \{P_1,P_2, P_3,P_4\}}$ yields a curve isomorphic to $\mathbb{G}_m$ (corresponding to $g^1_2\setminus \{P_1+P_2,P_3+P_4\}$).  Under the isomorphism above, we let $C\cong \mathbb{G}_m\subset D_5\setminus \cup_{j=1}^4D_j$ be the corresponding curve in the boundary of $\Sym^3_X$ (in $\Sym^3_{\overline{X}}$).

  From the proof of Theorem \ref{thm:greenconverse} (see \cite[Th.~7.2.13]{NWBook}), for any two points $Q,R\in C$ there are sequences $\{Q_i\}$ and $\{R_i\}$ in $\Sym^3_X$ converging to $Q$ and $R$, respectively, such that $$d_{\Sym^3_X}(Q_i,R_i)\to 0.$$  It follows immediately that if $\Sym^3_X$ is hyperbolically embedded in $Y$, then the morphism $\Sym^3_{\Xbar} \to Y$ must contract $C$ to a point. (In particular, the presence of this $\mathbb{G}_m$ is enough to conclude that $\Sym^3_X\subset \Sym^3_{\Xbar}$ is not a hyperbolic embedding and, as we will show now, its presence is also enough   to conclude that $\Sym^3_X\subset Y$ is not a hyperbolic embedding.)

Fix a Weierstrass point $P_0$  in $\Xbar$ and 
consider  the morphism $f:\Sym^3_{\Xbar}\to \Jac(\Xbar)$ given by  $(P,Q,R)\mapsto [P+Q+R-3P_0]$.  Furthermore,  consider the embedding of $\Xbar$ in $\Jac(\Xbar)$ given by $P\mapsto [P-P_0]$ and identify $\Xbar$ with its image.  Let $S=f^{-1}(\Xbar)$, and note that $S$ is a  surface in $\Sym^3_{\Xbar}$.  Explicitly, the surface $S$ consists of points of the form $P+\iota(P)+Q$ in $\Sym^3_{\Xbar}$ (identifying points with degree $3$ effective divisors).  The closure of the curve $C$ is precisely the fiber $F$ of $S$ above $[P_5-P_0]$.  Then, by restriction, we obtain a morphism $S\to Y$ which contracts the fiber $F$ (of the morphism $S\to\overline{X}$). In a diagram:
\[
\xymatrix{ S \ar[rrrr]^{\textrm{contracts the fiber } F} \ar[d]  & & & & Y   \\ 
			\overline{X} & &  &		 &						 }
\]
 Since the morphism $S\to Y$ contracts the fiber $F$, it follows from the Rigidity Lemma  \cite[Lemma 1.6]{KM} that  it contracts every fiber.  However, since $S\to Y$ is birational onto its image, we obtain a contradiction, and conclude that $\Sym^3_X$ is not hyperbolically embeddable.
\end{proof}

\begin{proof}[Proof of Theorem \ref{thm:finale2}]
Let $\overline{X}$ be a smooth projective connected hyperelliptic curve of genus $g\geq 3$ over $\Qbar$, and let $\iota:\overline{X}\to \overline{X}$ be the hyperelliptic involution. Let $P_1, P_3, P_5$ be pairwise distinct non-Weierstrass points, and   let $P_2 :=\iota(P_1)$ and $P_4 := \iota(P_3)$. Define $X:= \overline{X}\setminus \{P_1,P_2, P_3,P_4,P_5\}$.  From the proof of Theorem \ref{thm:finale}, the variety $\Sym^3_X$ is arithmetically hyperbolic over $\Qbar$ and $\Sym^3_{X_\sigma}$ is Brody hyperbolic for every $\sigma:\Qbar\to \CC$.

Let $Y$ be a projective variety  over $\mathbb{C}$ and let $\Sym^3_{X_\sigma}\subset Y$ be an embedding.  By Theorem \ref{thm:final_hypemb},  this embedding $\Sym^3_{X_\sigma}\subset Y$ is not a hyperbolic embedding. This concludes the proof. 
\end{proof}

 \bibliography{refsci}{}

\def\cprime{$'$}
\begin{thebibliography}{ACGH85}

\bibitem[ACGH85]{ACGH}
E.~Arbarello, M.~Cornalba, P.~A. Griffiths, and J.~Harris.
\newblock {\em Geometry of algebraic curves. {V}ol. {I}}, volume 267 of {\em
  Grundlehren der Mathematischen Wissenschaften [Fundamental Principles of
  Mathematical Sciences]}.
\newblock Springer-Verlag, New York, 1985.

\bibitem[AH91]{AbrHar}
D.~Abramovich and J.~Harris.
\newblock Abelian varieties and curves in {$W_d(C)$}.
\newblock {\em Compositio Math.}, 78(2):227--238, 1991.

\bibitem[Aut09]{Autissier1}
P.~Autissier.
\newblock G\'eom\'etries, points entiers et courbes enti\`eres.
\newblock {\em Ann. Sci. \'Ec. Norm. Sup\'er. (4)}, 42(2):221--239, 2009.

\bibitem[Aut11]{Autissier2}
P.~Autissier.
\newblock Sur la non-densit\'e des points entiers.
\newblock {\em Duke Math. J.}, 158(1):13--27, 2011.

\bibitem[AW07]{AnWang}
T.~T.~H. An and J.~T.-Y. Wang.
\newblock An effective {S}chmidt's subspace theorem for non-linear forms over
  function fields.
\newblock {\em J. Number Theory}, 125(1):210--228, 2007.

\bibitem[BJK]{vBJK}
R.~van Bommel, A.~Javanpeykar, and L.~Kamenova.
\newblock Boundedness in families with applications to arithmetic
  hyperbolicity.
\newblock {\em arXiv:1907.11225}.

\bibitem[BM86]{BM}
W.~D. Brownawell and D.~W. Masser.
\newblock Vanishing sums in function fields.
\newblock {\em Math. Proc. Cambridge Philos. Soc.}, 100(3):427--434, 1986.

\bibitem[CLZ09]{CLZ}
P.~Corvaja, A.~Levin, and U.~Zannier.
\newblock Integral points on threefolds and other varieties.
\newblock {\em Tohoku Math. J. (2)}, 61(4):589--601, 2009.

\bibitem[CZ02]{CZ2}
P.~Corvaja and U.~Zannier.
\newblock A subspace theorem approach to integral points on curves.
\newblock {\em C. R. Math. Acad. Sci. Paris}, 334(4):267--271, 2002.

\bibitem[CZ04]{CZ3}
P.~Corvaja and U.~Zannier.
\newblock On integral points on surfaces.
\newblock {\em Ann. of Math. (2)}, 160(2):705--726, 2004.

\bibitem[CZ06]{CZ4}
Pietro Corvaja and Umberto Zannier.
\newblock On the integral points on certain surfaces.
\newblock {\em Int. Math. Res. Not.}, pages Art. ID 98623, 20, 2006.

\bibitem[CZ10]{CZ}
P.~Corvaja and U.~Zannier.
\newblock Integral points, divisibility between values of polynomials and
  entire curves on surfaces.
\newblock {\em Adv. Math.}, 225(2):1095--1118, 2010.

\bibitem[Dem97]{Demailly}
J.-P. Demailly.
\newblock Algebraic criteria for {K}obayashi hyperbolic projective varieties
  and jet differentials.
\newblock In {\em Algebraic geometry---{S}anta {C}ruz 1995}, volume~62 of {\em
  Proc. Sympos. Pure Math.}, pages 285--360. Amer. Math. Soc., Providence, RI,
  1997.

\bibitem[Fal83]{Faltings2}
G.~Faltings.
\newblock Endlichkeitss\"atze f\"ur abelsche {V}ariet\"aten \"uber
  {Z}ahlk\"orpern.
\newblock {\em Invent. Math.}, 73(3):349--366, 1983.

\bibitem[Fal84]{FaltingsComplements}
G.~Faltings.
\newblock Complements to {M}ordell.
\newblock In {\em Rational points ({B}onn, 1983/1984)}, Aspects Math., E6,
  pages 203--227. Vieweg, Braunschweig, 1984.

\bibitem[Fal94]{FaltingsLang}
G.~Faltings.
\newblock The general case of {S}. {L}ang's conjecture.
\newblock In {\em Barsotti {S}ymposium in {A}lgebraic {G}eometry ({A}bano
  {T}erme, 1991)}, volume~15 of {\em Perspect. Math.}, pages 175--182. Academic
  Press, San Diego, CA, 1994.

\bibitem[Fuj72]{Fuji}
H.~Fujimoto.
\newblock On holomorphic maps into a taut complex space.
\newblock {\em Nagoya Math. J.}, 46:49--61, 1972.

\bibitem[Gre77]{GreenHyp}
M.~L. Green.
\newblock The hyperbolicity of the complement of {$2n+1$} hyperplanes in
  general position in {$P\sb{n}$} and related results.
\newblock {\em Proc. Amer. Math. Soc.}, 66(1):109--113, 1977.

\bibitem[HL19]{HL}
G.~Heier and A.~Levin.
\newblock On the degeneracy of integral points and entire curves in the
  complement of nef effective divisors (preprint).
\newblock 2019.

\bibitem[Jav]{JBook}
A.~Javanpeykar.
\newblock The {L}ang-{V}ojta conjectures on projective pseudo-hyperbolic
  varieties.
\newblock {\em Preprint}.

\bibitem[{Jav}18]{JAut}
A.~{Javanpeykar}.
\newblock Arithmetic hyperbolicity: endomorphisms, automorphisms, hyperkahler
  varieties, geometricity.
\newblock {\em arXiv:1809.06818}, 2018.

\bibitem[JKa]{JKa}
A.~Javanpeykar and L.~Kamenova.
\newblock Demailly's notion of algebraic hyperbolicity: geometricity,
  boundedness, moduli of maps.
\newblock {\em Mathematische Zeitschrift, volume 296, 1645-1672, (2020)}.

\bibitem[JKb]{JKuch}
A.~Javanpeykar and R.~A. Kucharczyk.
\newblock Algebraicity of analytic maps to a hyperbolic variety.
\newblock {\em Math. Nachrichten, 293 (2020), no. 8 (August)}.

\bibitem[JL]{JLitt}
A.~Javanpeykar and D.~Litt.
\newblock Integral points on algebraic subvarieties of period domains: from
  number fields to finitely generated fields.
\newblock {\em arXiv:1907.13536}.

\bibitem[JL17]{JL}
A.~Javanpeykar and D.~Loughran.
\newblock Complete intersections: moduli, {T}orelli, and good reduction.
\newblock {\em Math. Ann.}, 368(3-4):1191--1225, 2017.

\bibitem[JL18]{JLFano}
A.~Javanpeykar and D.~Loughran.
\newblock Good reduction of {F}ano threefolds and sextic surfaces.
\newblock {\em Ann. Sc. Norm. Super. Pisa Cl. Sci. (5)}, 18(2):509--535, 2018.

\bibitem[JSZ]{JSZ}
A.~Javanpeykar, R.~Sun, and K.~Zuo.
\newblock The {S}hafarevich conjecture revisited: Finiteness of pointed
  families of polarized varieties.
\newblock {\em arXiv:2005.05933}.

\bibitem[JV]{JV}
A.~Javanpeykar and A.~Vezzani.
\newblock Non-archimedean hyperbolicity and applications.
\newblock {\em arXiv:1808.09880}.

\bibitem[JX]{JXie}
A.~Javanpeykar and J.~Xie.
\newblock Finiteness properties of pseudo-hyperbolic varieties.
\newblock {\em IMRN, to appear. arXiv:1909.12187}.

\bibitem[KL11]{KovacsLieblich}
S.J. Kov{\'a}cs and M.~Lieblich.
\newblock Erratum for {B}oundedness of families of canonically polarized
  manifolds: a higher dimensional analogue of {S}hafarevich's conjecture.
\newblock {\em Ann. of Math. (2)}, 173(1):585--617, 2011.

\bibitem[KM98]{KM}
J.~Koll\'{a}r and S.~Mori.
\newblock {\em Birational geometry of algebraic varieties}, volume 134 of {\em
  Cambridge Tracts in Mathematics}.
\newblock Cambridge University Press, Cambridge, 1998.
\newblock With the collaboration of C. H. Clemens and A. Corti, Translated from
  the 1998 Japanese original.

\bibitem[Kob98]{KobayashiBook}
S.~Kobayashi.
\newblock {\em Hyperbolic complex spaces}, volume 318 of {\em Grundlehren der
  Mathematischen Wissenschaften [Fundamental Principles of Mathematical
  Sciences]}.
\newblock Springer-Verlag, Berlin, 1998.

\bibitem[Kwa69]{Kwack}
M.~H. Kwack.
\newblock Generalization of the big {P}icard theorem.
\newblock {\em Ann. of Math. (2)}, 90:9--22, 1969.

\bibitem[Lan60]{LangIHES}
S.~Lang.
\newblock Integral points on curves.
\newblock {\em Inst. Hautes \'{E}tudes Sci. Publ. Math.}, (6):27--43, 1960.

\bibitem[Lan74]{Lang1}
S.~Lang.
\newblock Higher dimensional {D}iophantine problems.
\newblock {\em Bull. Amer. Math. Soc.}, 80:779--787, 1974.

\bibitem[Lan86]{Lang2}
S.~Lang.
\newblock Hyperbolic and {D}iophantine analysis.
\newblock {\em Bull. Amer. Math. Soc. (N.S.)}, 14(2):159--205, 1986.

\bibitem[Lan87]{LangHyperbolic}
S.~Lang.
\newblock {\em Introduction to complex hyperbolic spaces}.
\newblock Springer-Verlag, New York, 1987.

\bibitem[Lan97]{LangSurvey}
S.~Lang.
\newblock {\em Survey of Diophantine Geometry}.
\newblock Springer-Verlag, New York, 1997.

\bibitem[Lev09]{LevinAnnals}
A.~Levin.
\newblock Generalizations of {S}iegel's and {P}icard's theorems.
\newblock {\em Ann. of Math. (2)}, 170(2):609--655, 2009.

\bibitem[Lev16]{Levin}
A.~Levin.
\newblock Integral points of bounded degree on affine curves.
\newblock {\em Compos. Math.}, 152(4):754--768, 2016.

\bibitem[Lic]{Licht}
P.~Licht.
\newblock Finiteness theorems for complements of large divisors.
\newblock {\em Master's thesis, University of Mainz}.

\bibitem[Mah33]{Mahler1}
K.~Mahler.
\newblock Zur {A}pproximation algebraischer {Z}ahlen. {I}.
\newblock {\em Math. Ann.}, 107(1):691--730, 1933.

\bibitem[Mas86]{Mason}
R.~C. Mason.
\newblock Norm form equations. {I}.
\newblock {\em J. Number Theory}, 22(2):190--207, 1986.

\bibitem[Mor95]{Moriwaki}
A.~Moriwaki.
\newblock Remarks on rational points of varieties whose cotangent bundles are
  generated by global sections.
\newblock {\em Math. Res. Lett.}, 2(1):113--118, 1995.

\bibitem[Nit05]{Nitsure}
N.~Nitsure.
\newblock Construction of {H}ilbert and {Q}uot schemes.
\newblock In {\em Fundamental algebraic geometry}, volume 123 of {\em Math.
  Surveys Monogr.}, pages 105--137. Amer. Math. Soc., Providence, RI, 2005.

\bibitem[NW02]{NW}
J.~Noguchi and J.~Winkelmann.
\newblock Holomorphic curves and integral points off divisors.
\newblock {\em Math. Z.}, 239(3):593--610, 2002.

\bibitem[NW14]{NWBook}
J.~Noguchi and J.~Winkelmann.
\newblock {\em Nevanlinna theory in several complex variables and {D}iophantine
  approximation}, volume 350 of {\em Grundlehren der Mathematischen
  Wissenschaften [Fundamental Principles of Mathematical Sciences]}.
\newblock Springer, Tokyo, 2014.

\bibitem[Par50]{Parry}
C.~J. Parry.
\newblock The {$\mathfrak{p}$}-adic generalisation of the {T}hue-{S}iegel
  theorem.
\newblock {\em Acta Math.}, 83:1--100, 1950.

\bibitem[PR07]{PacienzaRousseau}
G.~Pacienza and E.~Rousseau.
\newblock On the logarithmic {K}obayashi conjecture.
\newblock {\em J. Reine Angew. Math.}, 611:221--235, 2007.

\bibitem[RW12]{RuWang}
M.~Ru and J.~T.-Y. Wang.
\newblock An effective {S}chmidt's subspace theorem for projective varieties
  over function fields.
\newblock {\em Int. Math. Res. Not. IMRN}, (3):651--684, 2012.

\bibitem[Sie21]{Siegel1}
C.~Siegel.
\newblock Approximation algebraischer {Z}ahlen.
\newblock {\em Math. Z.}, 10(3-4):173--213, 1921.

\bibitem[Sie14]{Siegel2}
C.~Siegel.
\newblock \"{U}ber einige {A}nwendungen diophantischer {A}pproximationen
  [reprint of {A}bhandlungen der {P}reu\ss ischen {A}kademie der
  {W}issenschaften. {P}hysikalisch-mathematische {K}lasse 1929, {N}r. 1].
\newblock In {\em On some applications of {D}iophantine approximations},
  volume~2 of {\em Quad./Monogr.}, pages 81--138. Ed. Norm., Pisa, 2014.

\bibitem[SZ00]{ShiffmanZaidenberg}
B.~Shiffman and M.~Zaidenberg.
\newblock Two classes of hyperbolic surfaces in {${\bf P}^3$}.
\newblock {\em Internat. J. Math.}, 11(1):65--101, 2000.

\bibitem[Ull04]{UllmoShimura}
E.~Ullmo.
\newblock Points rationnels des vari\'et\'es de {S}himura.
\newblock {\em Int. Math. Res. Not.}, (76):4109--4125, 2004.

\bibitem[Ura79]{Urata}
T.~Urata.
\newblock Holomorphic mappings into taut complex analytic spaces.
\newblock {\em T\^{o}hoku Math. J. (2)}, 31(3):349--353, 1979.

\bibitem[Voj89]{VojtaSub}
P.~Vojta.
\newblock A refinement of {S}chmidt's subspace theorem.
\newblock {\em Amer. J. Math.}, 111(3):489--518, 1989.

\bibitem[Voj99]{Vojta2}
P.~Vojta.
\newblock Integral points on subvarieties of semiabelian varieties. {II}.
\newblock {\em Amer. J. Math.}, 121(2):283--313, 1999.

\bibitem[Voj15]{VojtaLangExc}
P.~Vojta.
\newblock A {L}ang exceptional set for integral points.
\newblock In {\em Geometry and analysis on manifolds}, volume 308 of {\em
  Progr. Math.}, pages 177--207. Birkh\"{a}user/Springer, Cham, 2015.

\bibitem[Vol85]{Voloch}
J.~F. Voloch.
\newblock Diagonal equations over function fields.
\newblock {\em Bol. Soc. Brasil. Mat.}, 16(2):29--39, 1985.

\bibitem[Wan04]{Wang}
J.~T.-Y. Wang.
\newblock An effective {S}chmidt's subspace theorem over function fields.
\newblock {\em Math. Z.}, 246(4):811--844, 2004.

\bibitem[Zai85]{Zaidenberg1}
M.~G. Zaidenberg.
\newblock Hyperbolic imbedding of complements to divisors and the limit
  behavior of the {K}obayashi-{R}oyden metric.
\newblock {\em Mat. Sb. (N.S.)}, 127(169)(1):55--71, 143, 1985.

\bibitem[Zai86]{Zaidenberg2}
M.~G. Zaidenberg.
\newblock Criteria for hyperbolicity of imbedding of complements to
  hypersurfaces.
\newblock {\em Uspekhi Mat. Nauk}, 41(1(247)):191--192, 1986.

\end{thebibliography}
\bibliographystyle{alpha}

\end{document}